\documentclass[11pt]{amsart}

\usepackage{amsfonts,amssymb,amsmath,amsthm,amscd,enumerate}
\usepackage{latexsym}
\usepackage{euscript}

\newtheorem{theorem}{Theorem}[section]

\newtheorem{lemma}[theorem]{Lemma}
\newtheorem{corollary}[theorem]{Corollary}
\newtheorem{proposition}[theorem]{Proposition}
\newtheorem{remark}[theorem]{Remark}

\title{The additive group of a Lie nilpotent associative ring}

\author{Alexei Krasilnikov}

\address{Departamento de Matem\'atica, Universidade de
Bras\'\i lia, 70910-900 Bras\'\i lia, DF, Brazil}

\email{alexei@unb.br}

\date{}

\begin{document}

\maketitle

\begin{abstract}
Let $\mathbb Z \langle X \rangle$ be the free unitary associative
ring freely generated by an infinite countable set $X = \{
x_1,x_2, \dots \}$. Define a left-normed commutator $[x_1,x_2,
\dots , x_n]$ by $[a,b] = ab - ba$, $[a,b,c] = [[a,b],c]$. For $n
\ge 2$, let $T^{(n)}$ be the two-sided ideal in $\mathbb Z \langle
X \rangle$ generated by all commutators $[a_1,a_2, \dots , a_n]$
$( a_i \in \mathbb Z \langle X \rangle )$. It can be easily seen
that the additive group of the quotient ring $\mathbb Z \langle X
\rangle /T^{(2)}$ is a free abelian group. Recently Bhupatiraju,
Etingof, Jordan, Kuszmaul and Li have noted that the additive
group of $\mathbb Z \langle X \rangle /T^{(3)}$ is also free
abelian. In the present note we show that this is not the case for
$\mathbb Z \langle X \rangle /T^{(4)}$. More precisely, let
$T^{(3,2)}$ be the ideal in $\mathbb Z \langle X \rangle$
generated by $T^{(4)}$ together with all elements $[a_1, a_2, a_3]
[a_4, a_5]$ $(a_i \in \mathbb Z \langle X \rangle )$. We prove
that $T^{(3,2)}/T^{(4)}$ is a non-trivial elementary abelian
$3$-group and the additive group of $\mathbb Z \langle X \rangle
/T^{(3,2)}$ is free abelian.
\end{abstract}

\section{Introduction}

Let $\mathbb Z$ be the ring of integers and let $\mathbb Z \langle
X \rangle$ be the free unitary associative ring on the set $X = \{
x_i \mid i \in \mathbb N \}$. Then $\mathbb Z \langle X \rangle$
is the free $\mathbb Z$-module with a basis $\{ x_{i_1} x_{i_2}
\dots x_{i_k} \mid k \ge 0, i_l \in \mathbb N \}$ formed by the
non-commutative monomials in $x_1,x_2, \dots .$ Define a
left-normed commutator $[x_1,x_2, \dots , x_n]$ by $[a,b] = ab -
ba$, $[a,b,c] = [[a,b],c]$. For $n \ge 2$, let $T^{(n)}$ be the
two-sided ideal in $\mathbb Z \langle X \rangle$ generated by all
commutators $[a_1,a_2, \dots , a_n]$ $( a_i \in \mathbb Z \langle
X \rangle )$.

It is clear that the quotient ring $\mathbb Z \langle X \rangle
/T^{(2)}$ is isomorphic to the ring $\mathbb Z [X]$ of commutative
polynomials in $x_1,x_2, \dots$.  Hence, the additive group of
$\mathbb Z \langle X \rangle /T^{(2)}$ is a free abelian group.
Recently Bhupatiraju, Etingof, Jordan, Kuszmaul and Li
\cite{BEJKL12} have noted that the additive group of $\mathbb Z
\langle X \rangle /T^{(3)}$ is also free abelian. In the
present note we show that this is not the case for $\mathbb Z
\langle X \rangle /T^{(4)}$.

Our first result is as follows.

\begin{theorem}\label{theorem1}
Let $v = [x_1,x_2,x_3][x_4,x_5]$. Then $3 \, v \in T^{(4)}$ but $v
\notin T^{(4)}$.
\end{theorem}

Let $T^{(3,2)}$ be the two-sided ideal of the ring $\mathbb Z
\langle X \rangle$ generated by all elements $[a_1, a_2, a_3,
a_4]$ and $[a_1, a_2, a_3] [a_4, a_5]$ where $a_i \in \mathbb Z
\langle X \rangle$. Clearly, $T^{(4)} \subset T^{(3,2)}$.

Note that the ideal $T^{(4)}$ is closed under all substitutions
$x_i \rightarrow a_i$ $(i \in \mathbb N , a_i \in \mathbb Z
\langle X \rangle)$. By Theorem \ref{theorem1} we have $3 \, [a_2,
a_3, a_4] [a_5, a_6] \in T^{(4)}$; it follows that $3 \, a_1 [a_2,
a_3, a_4] [a_5, a_6] a_7 \in T^{(4)}$ for all $a_i \in \mathbb Z
\langle X \rangle$. Thus, we have

\begin{corollary}
$T^{(3,2)}/T^{(4)}$ is a non-trivial elementary abelian $3$-group.
\end{corollary}

Our second result is as follows.

\begin{theorem}\label{ZX/T32}
The additive group of the quotient ring $\mathbb Z \langle X
\rangle /T^{(3,2)}$ is free abelian.
\end{theorem}

Thus, the additive group of the ring $\mathbb Z \langle X \rangle
/T^{(4)}$ is a direct sum of a non-trivial elementary abelian
$3$-group $T^{(3,2)}/T^{(4)}$ and a free abelian group isomorphic
to $\mathbb Z \langle X \rangle /T^{(3,2)}$.

Let $X_m = \{ x_1, \dots ,x_m \} \subset X $; then $\mathbb Z
\langle X_m \rangle \subset \mathbb Z \langle X \rangle$. Let
$T^{(4)}_m = \mathbb Z \langle X_m \rangle \cap T^{(4)}$,
$T^{(3,2)}_m = \mathbb Z \langle X_m \rangle \cap T^{(3,2)}$. By
Theorem \ref{theorem1}, $T^{(4)}_m \subsetneqq T^{(3,2)}_m$ if $m
\ge 5$.

\begin{proposition}\label{T44T324}
If $m = 2, 3, 4$ then $T^{(4)}_m = T^{(3,2)}_m$. In particular,
for $m \le 4$ the additive group of $\mathbb Z \langle X_m \rangle
/ T^{(4)}_m$ is free abelian.
\end{proposition}

Define $\gamma_n = \gamma_n (\mathbb Z \langle X \rangle )$ by
$\gamma_1 = \mathbb Z \langle X \rangle$, $\gamma_{n+1} =
[\gamma_n, \mathbb Z \langle X \rangle ]$ $(n \ge 1)$. Then
$\gamma_n$ is the $n$-th term of the lower central series of
$\mathbb Z \langle X \rangle$ viewed as a Lie ring. Clearly,
$T^{(n)}$ is the two-sided ideal of $\mathbb Z \langle X \rangle$
generated by $\gamma_n$.

\begin{proposition}\label{gamma3gamma4}
Let $w = [x_1 [x_2, x_3, x_4], x_5]$. Then $w \in \gamma_3$, $6 \,
w \in \gamma_4$ but $w \notin \gamma_4$.
\end{proposition}

Thus, $w + \gamma_4$ is a non-trivial element of finite order
(dividing 6) of the additive group of the quotient $\gamma_3
/\gamma_4$.

Let $f \in \mathbb Z \langle X \rangle$ be a multihomogeneous
polynomial. It was conjectured in \cite{BEJKL12} that if
$f + \gamma_{l+1}$ is a torsion element in $\gamma_l / \gamma_{l+1}$
then the degree of $f$ is at least $l+3$ (Conjecture 5.3) and the
degree of $f$ with respect to each generator $x_j$ is a multiple
of the order of $f + \gamma_{l+1}$ (Conjecture 5.2). Since $w$ is of
degree $5$ and has degree $1$ with respect to each $x_i$, $1 \le i
\le 5$, Proposition \ref{gamma3gamma4} gives a counter-example to
these conjectures.

\begin{proof}[Proof of Proposition \ref{gamma3gamma4}]
By the identity (14) of \cite{FS07}, we have $w \in \gamma_3$. It
can be deduced from the proof of \cite[Lemma 6.1]{BJ10} that $6 w
\in \gamma_4$. On the other hand, $w = x_1 [x_2, x_3, x_4, x_5] +
[x_1,x_5] [x_2, x_3, x_4] \notin T^{(4)}$ because $x_1 [x_2, x_3,
x_4, x_5] \in T^{(4)}$ and, by Theorem \ref{theorem1}, $[x_1,x_5]
[x_2, x_3, x_4] \notin T^{(4)}$. Since $\gamma_4 \subset T^{(4)}$,
we have $w \notin \gamma_4$, as required.
\end{proof}

\begin{remark}
{ \rm
The proof of Lemma \ref{in} below shows that the reason behind the
existence of $3$-torsion in the additive group of $\mathbb Z
\langle X \rangle / T^{(4)}$ is the Jacobi identity. For this
reason one might expect the structure of the additive group
of $\mathbb Z \langle X \rangle / T^{(n)}$ for arbitrary $n > 4$
to be similar to that of $\mathbb Z \langle X \rangle / T^{(4)}$, that is,
there should be an ideal $I^{(n)} \subset \mathbb Z \langle X \rangle$
such that $T^{(n)} \subset I^{(n)}$, the quotient $I^{(n)} /
T^{(n)}$ being an elementary abelian $3$-group (possibly trivial for
some $n$) and the additive group of $\mathbb Z \langle X \rangle /
I^{(n)}$ being free abelian. This might also suggest that the
counter-example to Conjectures 5.2 and 5.3 of \cite{BEJKL12} given
in Proposition \ref{gamma3gamma4} is in a certain sense
exceptional, and that it may be possible to modify the conjectures
slightly so that they would be true.
}
\end{remark}

\section{Proof of Theorem 1.1}

The following lemma is well-known (see, for instance,
\cite[Theorem 3.4]{EKM09}, \cite[Lemma 1]{Gordienko07},
\cite[Lemma 2]{Latyshev65}) but we prove it here in order to have
the paper more self-contained.
\begin{lemma}\label{32+32T4}
For all $a_1, \dots ,a_5 \in \mathbb Z \langle X \rangle$,
\begin{equation}\label{32+32T4-1}
[a_{1},a_{2},a_{3}][a_{4},a_{5}] +
[a_{1},a_{2},a_{4}][a_{3},a_{5}] \in T^{(4)},
\end{equation}
\[
[a_{1},a_{2},a_{3}][a_{4},a_{5}] +
[a_{1},a_{4},a_{3}][a_{2},a_{5}] \in T^{(4)}.
\]
\end{lemma}

\begin{proof}
Since $[a,bc]=b[a,c] + [a,b]c$, $[ab,c] = a[b,c] + [a,c]b$, we
have
\begin{multline*}
[a_1,a_2,a_3a_4,a_5] = [a_3[a_1,a_2,a_4] + [a_1,a_2,a_3]a_4, a_5]
\\
= a_3[a_1,a_2,a_4,a_5] + [a_3,a_5][a_1,a_2,a_4] +
[a_1,a_2,a_3][a_4,a_5] + [a_1,a_2,a_3,a_5]a_4 .
\end{multline*}
It is clear that $[a_1,a_2,a_3a_4,a_5],\ a_3[a_1,a_2,a_4,a_5],\
[a_1,a_2,a_3,a_5]a_4  \in T^{(4)}$ so $[a_3,a_5][a_1,a_2,a_4] +
[a_1,a_2,a_3][a_4,a_5] \in T^{(4)}$. Further, $[[a_1,a_2,a_4],
[a_3,a_5]] = [a_1,a_2,a_4,a_3,a_5] - [a_1,a_2,a_4,a_5,a_3] \in
T^{(4)}$ so
\begin{multline*}
[a_{1},a_{2},a_{3}][a_{4},a_{5}] +
[a_{1},a_{2},a_{4}][a_{3},a_{5}] \\
= [a_3,a_5][a_1,a_2,a_4] + [a_1,a_2,a_3][a_4,a_5] +
[[a_1,a_2,a_4], [a_3,a_5]] \in T^{(4)},
\end{multline*}
as required.

Note also that
\begin{multline*}
[a_5,a_2a_4,a_1,a_3] = [a_2[a_5,a_4] + [a_5,a_2]a_4, a_1, a_3] \\
= [a_2[a_5,a_4,a_1] + [a_2,a_1][a_5,a_4] + [a_5,a_2][a_4,a_1] +
[a_5,a_2,a_1]a_4 , a_3] \\
= a_2[a_5,a_4,a_1,a_3] + [a_2,a_3][a_5,a_4,a_1] +
[a_2,a_1][a_5,a_4,a_3] + [a_2,a_1,a_3][a_5,a_4] \\
+ [a_5,a_2][a_4,a_1,a_3] + [a_5,a_2,a_3][a_4,a_1] +
[a_5,a_2,a_1][a_4,a_3] + [a_5,a_2,a_1,a_3]a_4.
\end{multline*}
It is clear that $[a_5,a_2a_4,a_1,a_3], a_2[a_5,a_4,a_1,a_3],
[a_5,a_2,a_1,a_3]a_4 \in T^{(4)}$. Also, by (\ref{32+32T4-1}),
$[a_2,a_3][a_5,a_4,a_1] + [a_2,a_1][a_5,a_4,a_3] \in T^{(4)}$ and 
$[a_5,a_2,a_3][a_4,a_1] + [a_5,a_2,a_1][a_4,a_3] \in T^{(4)}$. It
follows that $[a_2,a_1,a_3][a_5,a_4] + [a_5,a_2][a_4,a_1,a_3] \in
T^{(4)}$, therefore $[a_1,a_2,a_3][a_4,a_5] + [a_1,a_4,a_3]
[a_2,a_5] \in T^{(4)}$, as required.
\end{proof}

The following lemma is also well-known (see, for instance,
\cite{EKM09}, \cite[Lemma 1]{Gordienko07}, \cite[Lemma
1]{Volichenko78}).

\begin{lemma}\label{in}
For all $a_1, \dots ,a_5 \in \mathbb Z \langle X \rangle$, we 
have $3 \, [a_1,a_2,a_3][a_4,a_5] \in T^{(4)}$.
\end{lemma}

\begin{proof}
It is clear that if $\sigma = (12)$ or $\sigma = (45)$ then, for
all $a_1, \dots , a_5 \in \mathbb Z \langle X \rangle$,
\[
[a_{1}, a_{2}, a_{3}] [a_{4}, a_{5}] = - [a_{\sigma (1)},
a_{\sigma (2)}, a_{\sigma (3)}] [a_{\sigma (4)}, a_{\sigma (5)}] .
\]
On the other hand, if $\sigma = (34)$ or $\sigma = (24)$ then, by
Lemma \ref{32+32T4},
\[
[a_{1}, a_{2}, a_{3}] [a_{4}, a_{5}] \equiv - [a_{\sigma (1)},
a_{\sigma (2)}, a_{ \sigma (3)}] [a_{\sigma (4)}, a_{\sigma (5)}]
\pmod{T^{(4)}} .
\]
Since the transpositions $(12)$, $(45)$, $(34)$ and $(24)$
generate the entire group $S_5$ of the permutations of the set $\{
1,2,3,4,5 \}$, for all $\sigma \in S_5$ we have
\begin{equation}\label{permutation2}
[a_{1}, a_{2}, a_{3}] [a_{4}, a_{5}] \equiv  \mbox{sgn} (\sigma )
[a_{\sigma (1)}, a_{\sigma (2)}, a_{\sigma (3)}] [a_{\sigma (4)},
a_{\sigma (5)}] \pmod{T^{(4)}} .
\end{equation}

Now note that, by the Jacobi identity,
\[
[a_1, a_2, a_3] [a_4, a_5] + [a_2, a_3, a_1] [a_4, a_5] + [a_3,
a_1, a_2] [a_4, a_5] = 0.
\]
By (\ref{permutation2}), this equality implies $ 3 \, [a_1, a_2,
a_3] [a_4, a_5] \in T^{(4)}$ for all $a_1, \dots , a_5 \in \mathbb
Z \langle X \rangle$, as required. 
\end{proof}

\begin{remark}
{\rm
It is known (see, for instance, \cite[Lemma 1]{Volichenko78}) that  
\begin{multline*}
[a_1, a_2, a_3] [a_4, \dots , a_n, a_{n+1}] \\ 
\equiv \mbox{sgn} (\sigma ) [a_{\sigma (1)}, a_{\sigma (2)}, a_{\sigma (3)}] [a_4, 
\dots , a_n, a_{\sigma (n+1)}] \pmod{T^{(n)}}
\end{multline*}
for each $n \ge 4$, all $a_i \in \mathbb Z \langle X \rangle$ and 
all permutations $\sigma$ of the set $\{ 1, 2, 3, (n+1) \}$.
The proof is similar to that of (\ref{permutation2}). It follows that 
\[
3 \, [a_1, a_2, a_3] [a_4, \dots , a_n, a_{n+1}] \in T^{(n)}
\]
for all $n \ge 4$ and all $a_i \in \mathbb Z \langle X \rangle$.
}
\end{remark}

By Lemma \ref{in}, we have $3 \, v \in T^{(4)}$. To prove Theorem 
\ref{theorem1} it remains to prove the following.

\begin{lemma}\label{notin}
$v \notin T^{(4)}$.
\end{lemma}

Recall that $X_5 = \{ x_1,x_2,x_3,x_4,x_5 \} \subset X $, $\mathbb
Z \langle X_5 \rangle \subset \mathbb Z \langle X \rangle$,
$T^{(4)}_5 = \mathbb Z \langle X_5 \rangle \cap T^{(4)}$. Note
that $v = [x_1,x_2,x_3] [x_4,x_5] \in \mathbb Z \langle X_5
\rangle$. Let $I$ be the ideal of $\mathbb Z \langle X_5 \rangle$
spanned by all monomials $x_{i_1} x_{i_2} \dots x_{i_k}$ $(1 \le
i_1, i_2, \dots , i_k \le 5)$ such that $i_r = i_s$ for some $r
\ne s$. In particular, if $k>5$ then $x_{i_1} x_{i_2} \dots
x_{i_k} \in I$. The following proposition will be proved in the
next sections.

\begin{proposition}\label{prop_main}
$v \notin T^{(4)}_5 +I$.
\end{proposition}

Lemma \ref{notin} is an immediate corollary of Proposition
\ref{prop_main} and Theorem \ref{theorem1} follows immediately
from Lemmas \ref{in} and \ref{notin}. This completes the proof of
Theorem \ref{theorem1} provided that Proposition \ref{prop_main}
is proved.

\section{Auxiliary results}

Let $P_n$ $(n \le 1)$ be the subgroup of the additive group of
$\mathbb Z \langle X \rangle$ generated by all monomials which are
of degree $1$ in each variable $x_1, \dots , x_n$ and do not
contain any other variable. Then $P_n$ is a free abelian group of
rank $n!$. Let $L(X)$ be the free Lie ring on the free generating
set $X$, $L(X) \subset \mathbb Z \langle X \rangle$. Define $V_n =
L(X) \cap P_n$. The following lemma is well known (see, for
instance, \cite[Exercise 4.3.8]{Drenskybook}).

\begin{lemma}\label{basis_V_n}
For each $n >1$, $V_n$ is a free abelian group with a basis
\begin{equation}\label{commutator}
\Big\{ [x_n, x_{i_1}, \dots , x_{i_{n-1}}] \mid \{ i_1, \dots ,
i_{n-1} \} = \{ 1, 2, \dots , n-1 \} \Big\} .
\end{equation}
\end{lemma}

\begin{proof}
It can be easily proved using the Jacobi identity that the
commutators of (\ref{commutator}) generate $V_n$ as a subgroup of
the additive group of $\mathbb Z \langle X \rangle$. On the other
hand, the leading term (in the lexicographic order) of the
commutator $[x_n, x_{i_1}, \dots , x_{i_{n-1}}]$ is the monomial
$x_n x_{i_1} \dots x_{i_{n-1}}$. Hence, distinct commutators of
(\ref{commutator}) have distinct leading terms. It follows that
the elements of (\ref{commutator}) are linearly independent so
they form a basis of $V_n$.
\end{proof}

Let $W_1$ be the subgroup of the additive group of $\mathbb Z
\langle X_5 \rangle$ generated by all elements $x_{i_1} [x_{i_2},
x_{i_3}, x_{i_4}, x_{i_5}]$ and let $W_2$ the subgroup generated
by the elements $[x_{i_1}, x_{i_2}, x_{i_3}, x_{i_4}, x_{i_5}]$
and $[x_{i_1}, x_{i_2}, x_{i_3}] [x_{i_4}, x_{i_5}]$ where $\{
i_1, i_2, i_3, i_4 , i_5 \} = \{ 1,2,3,4,5 \}$. Note that
\[
[x_{i_1}, x_{i_2}] [x_{i_3}, x_{i_4}, x_{i_5}]  = [x_{i_3},
x_{i_4}, x_{i_5}][x_{i_1}, x_{i_2}] + [[x_{i_1}, x_{i_2}],
[x_{i_3}, x_{i_4}, x_{i_5}]]
\]
so $ [x_{i_1}, x_{i_2}] [x_{i_3}, x_{i_4}, x_{i_5}]  \in W_2 $.

\begin{lemma}\label{W1W2}
$W_1 \cap W_2 = 0$.
\end{lemma}

\begin{proof}
Let $W_1^{(j)}$ $(1 \le j \le 5)$ be the subgroup of $W_1$
generated by the elements  $x_{j} [x_{i_2}, x_{i_3}, x_{i_4},
x_{i_5}]$ where $ \{ j, i_2, i_3, i_4, i_5 \} = \{1, 2, 3, 4, 5
\}$. It is clear that $W_1 = W_1^{(1)} \oplus \dots \oplus
W_1^{(5)}$.

Let $\nu_i$ be the endomorphism of $\mathbb Z \langle X \rangle$
defined by $\nu_i (x_i) = 1$, $\nu_i (x_j) = x_j$ for all $j \ne
i$. It is clear that $\nu_i (W_2) = 0$  and $\nu_i (W_1^{(j)}) =
0$ for all $j \ne i$. On the other hand, $W_1^{(i)} \cap \mbox{Ker
} \nu_i = 0$. Indeed, suppose in order to simplify notation that
$i=1$. Then it follows easily from Lemma \ref{basis_V_n} that the
set $ C_1 = \Big\{ x_1 [x_5, x_{i_2}, x_{i_3}, x_{i_4}] \mid \{
i_2, i_3, i_4 \} = \{ 2,3,4 \} \Big\}$ is a basis of the free
abelian group $W_1^{(1)}$. On the other hand, by the same lemma,
the set $ \nu_1 (C_1) = \Big\{ [x_5, x_{i_2}, x_{i_3}, x_{i_4}]
\mid \{ i_2, i_3, i_4 \} = \{ 2,3,4 \} \Big\}$ is linearly
independent. It follows that $W_1^{(1)} \cap \mbox{Ker } \nu_1 =
0$, as claimed. If $i>1$ then the proof is similar.

Now we are in a position to complete the proof of Lemma
\ref{W1W2}. Suppose that $f \in W_1 \cap W_2$. Since $f \in W_2$,
we have $f \in \mbox{Ker } \nu_i$ for all $i$. On the other hand,
$f \in W_1$ so $f = f_1 + \dots + f_5$ where $f_i \in W_1^{(i)}$.
For each $i$ we have $f_i \in \mbox{Ker } \nu_i$ because $f \in
\mbox{Ker } \nu_i$ and $f_j \in \mbox{Ker } \nu_i$ for all $j \ne
i$. Since $W_1^{(i)} \cap \mbox{Ker } \nu_i = 0$, we have $f_i =
0$. Thus, $f = 0$ so $W_1 \cap W_2 = 0$. The proof of Lemma
\ref{W1W2} is completed.
\end{proof}

Let
\begin{gather*}
{\mathcal B}_1 = \Big\{ [x_5, x_{i_1}, x_{i_2}, x_{i_3}, x_{i_4}]
\mid \{ i_1, i_2, i_3, i_4  \} = \{ 1,2,3,4 \} \Big\} ,
\\
{\mathcal B}_2 = \Big\{ [x_{i_1}, x_{i_2}, x_{i_3}] [x_{5},
x_{i_4}] \mid \{ i_1, i_2, i_3, i_4  \} = \{ 1,2,3,4 \} , i_1 >
i_2, i_3 \Big\} ,
\\
{\mathcal B}_3 = \Big\{ [x_{i_1}, x_{i_2}] [x_{5}, x_{i_3},
x_{i_4}] \mid \{ i_1, i_2, i_3, i_4  \} = \{ 1,2,3,4 \} , i_1 >
i_2 \Big\}
\end{gather*}
and let ${\mathcal B} = {\mathcal B}_1 \cup {\mathcal B}_2 \cup
{\mathcal B}_3$.

\begin{lemma}\label{W2basis}
$W_2$ is a free abelian group with a basis ${\mathcal B}$.
\end{lemma}

\begin{proof}
It can be easily seen that ${\mathcal B} \subset W_2$ and
${\mathcal B}$ generates $W_2$ as a subgroup of the additive group
of $\mathbb Z \langle X_5 \rangle$. On the other hand, the leading
monomials (in the lexicographic order) of distinct elements of
${\mathcal B}$ are distinct so the elements of ${\mathcal B}$ are
linearly independent. The result follows.
\end{proof}

Let $\phi : \mathbb Z \langle X_5 \rangle \rightarrow \mathbb Z
\langle X_5 \rangle /I$ be the natural epimorphism, $\phi (f) = f
+ I$ for all $f \in \mathbb Z \langle X_5 \rangle$. Let $W = W_1 +
W_2$ and let $U = \phi (W)$, $U_i = \phi (W_i)$ $(i = 1,2)$. Then
$\mbox{Ker } \phi \cap W = 0$ so $\left. \phi \right|_W : W
\rightarrow U$ is an isomorphism. Hence, Lemmas  \ref{W1W2} and
\ref{W2basis} imply the following assertions.

\begin{corollary}\label{U1U2}
$U_1 \cap U_2 = 0$.
\end{corollary}

\begin{corollary}\label{U2basis}
$U_2$ is a free abelian group with a basis $\{ b + I \mid b \in
{\mathcal B} \}$.
\end{corollary}

\section{Proof of Proposition 2.5}

It is clear that the additive group of the ring $\mathbb Z \langle
X_5 \rangle $ is a direct sum of the subgroups $R_i$ $( i \ge 0 )$
generated by the monomials of degree $i$,
\[
\mathbb Z \langle X_5 \rangle = \bigoplus_{i \ge 0} R_i .
\]
It is also clear that $T^{(4)}_5$ is generated as a subgroup of
the additive group of $\mathbb Z \langle X_5 \rangle $ by the
polynomials $a_1[a_2, a_3, a_4, a_5] a_6$ where $a_i$ $(1 \le i
\le 6)$ are monomials. It follows that $T^{(4)}_5 \cap R_i$ is
generated as an additive group by the polynomials above such that
$\sum_{j=1}^6 \mbox{deg} (a_j) = i$.

Let $\overline{R}_i = (R_i + I) /I$, $\overline{T}^{(4)}_5 =
(T^{(4)}_5 + I) /I$. It is clear that $\mathbb Z \langle X_5
\rangle /I = \overline{R}_0 \oplus \overline{R}_1 \oplus \ldots
\oplus \overline{R}_5$, $\overline{T}^{(4)}_5 = (\overline{R}_4
\cap \overline{T}^{(4)}_5) \oplus (\overline{R}_5 \cap
\overline{T}^{(4)}_5)$ and $\overline{R}_5 \cap
\overline{T}^{(4)}_5$ is generated as a subgroup of the additive
group of $\mathbb Z \langle X_5 \rangle /I$ by the elements
$a_1[a_2, a_3, a_4, a_5] a_6 + I$ where $a_i$ $(1 \le i \le 6)$
are monomials and $\sum_{j=1}^6 \mbox{deg} (a_j) = 5$. It follows
that $\overline{R}_5 \cap \overline{T}^{(4)}_5$ is generated as an
additive group by
\begin{equation}\label{generators}
\begin{split}
x_{i_1} [x_{i_2}, x_{i_3}, x_{i_4}, x_{i_5}] + I, \  & [x_{i_1},
x_{i_2}, x_{i_3}, x_{i_4}] x_{i_5} + I, \ [(x_{i_1} x_{i_2}),
x_{i_3}, x_{i_4}, x_{i_5}] + I, \\
[x_{i_1}, (x_{i_2}x_{i_3}), x_{i_4}, x_{i_5}] + I, \ &[x_{i_1},
x_{i_2}, (x_{i_3}x_{i_4}), x_{i_5}] + I, \
[x_{i_1},x_{i_2},x_{i_3},(x_{i_4}x_{i_5})] + I
\end{split}
\end{equation}
where $\{ i_1, i_2, i_3, i_4 , i_5 \} = \{ 1,2,3,4,5 \}$.

We claim that $\overline{R}_5 \cap \overline{T}^{(4)}_5$ is
generated as an additive group by the elements
\begin{equation}\label{generators1}
x_{i_1} [x_{i_2}, x_{i_3}, x_{i_4}, x_{i_5}] + I, \  [x_{i_1},
x_{i_2}, x_{i_3}, x_{i_4}, x_{i_5}] + I
\end{equation}
and
\begin{equation}\label{generators2}
\begin{split}
& [x_{i_1}, x_{i_2}, x_{i_3}] [x_{i_4}, x_{i_5}] + [x_{i_1},
x_{i_2}, x_{i_4}] [x_{i_3}, x_{i_5}] + I, \\
& [x_{i_1}, x_{i_2}, x_{i_3}] [x_{i_4}, x_{i_5}] + [x_{i_1},
x_{i_4}, x_{i_3}] [x_{i_2}, x_{i_5}] + I.
\end{split}
\end{equation}
Indeed, it is straightforward to check repeating the calculations
of the proof of Lemma \ref{32+32T4} that all elements
(\ref{generators}) belong to the additive group generated by
(\ref{generators1}) and (\ref{generators2}). On the other hand, it
is clear that all elements (\ref{generators1}) belong to
$\overline{R}_5 \cap \overline{T}^{(4)}_5$. By Lemma
\ref{32+32T4}, all elements (\ref{generators2}) belong to
$\overline{R}_5 \cap \overline{T}^{(4)}_5$ as well. Therefore, the
elements (\ref{generators1}) and (\ref{generators2}) generate
$\overline{R}_5 \cap \overline{T}^{(4)}_5$, as claimed.

Recall that $U_i = \phi (W_i) = (W_i + I)/I$ $(i=1,2)$, $U = U_1 +
U_2$. Since $\overline{R}_5 \cap \overline{T}^{(4)}_5$ is
generated by the elements (\ref{generators1}) and
(\ref{generators2}), we have $\overline{R}_5 \cap
\overline{T}^{(4)}_5 \subseteq U$.

Let $U_2'$ be the subgroup of $U_2$ generated by all elements
$[x_{i_1}, x_{i_2}, x_{i_3}, x_{i_4}, x_{i_5}] + I$. It follows
easily from Lemma \ref{basis_V_n} that $U_2'$ is a free abelian
group and the set $\{ b + I \mid b \in {\mathcal B}_1 \}$ is a
basis of $U_2'$. Then, by Corollary \ref{U2basis}, $U_2 / U_2'$ is
a free abelian group with a basis formed by the images of the
elements of ${\mathcal B}_2$ and ${\mathcal B}_3$.

Note that $U/(U_1 + U_2') = (U_1 + U_2)/(U_1 + U_2') \simeq
U_2/(U_1 + U_2') \cap U_2$. By Corollary \ref{U1U2}, $(U_1 + U_2')
\cap U_2 = U_2'$ so $U/(U_1 + U_2') \simeq U_2/U_2'$. Hence,
$U/(U_1 + U_2')$ is a free abelian group with a basis formed by
$\overline{{\mathcal B}}_2 \cup \overline{{\mathcal B}}_3$ where
$\overline{{\mathcal B}}_j = \{ b + (U_1 + U_2') \mid b \in
{\mathcal B}_j \}$ $(j = 2,3)$.

Note also that $U_1, U_2' \subset \overline{R}_5 \cap
\overline{T}^{(4)}_5$. Since $U_1 + U_2'$ is generated by the
elements (\ref{generators1}) and $\overline{R}_5 \cap
\overline{T}^{(4)}_5$ is generated by the elements
(\ref{generators1}) and (\ref{generators2}), the quotient group
$(\overline{R}_5 \cap \overline{T}^{(4)}_5)/(U_1+U_2')$ is
generated by the images of elements (\ref{generators2}).

Let $H$ be the free (additive) abelian group freely generated by
the set
\[
\Big\{ h_{i_1 i_2 i_3 i_4 i_5} \mid \{ i_1, i_2, i_3, i_4, i_5 \}
= \{ 1,2,3,4,5 \} \Big\} .
\]
Define a homomorphism $\psi : H \rightarrow U/(U_1+U_2')$ by
\[
\psi (h_{i_1 i_2 i_3 i_4 i_5}) = [x_{i_1}, x_{i_2}, x_{i_3}]
[x_{i_4}, x_{i_5}] + (U_1 + U_2').
\]
Let $Q$ be the subgroup of $H$ generated by the elements
\begin{equation*}
\begin{split}
h_{i_1 i_2 i_3 i_4 i_5} + h_{i_2 i_1 i_3 i_4 i_5}, \ h_{i_1 i_2
i_3 i_4 i_5} + h_{i_1 i_2 i_3 i_5 i_4}, \ h_{i_1 i_2 i_3 i_4 i_5}
+ h_{i_2 i_3 i_1 i_4 i_5} + h_{i_3 i_1 i_2 i_4 i_5}.
\end{split}
\end{equation*}

\begin{lemma}
$\mbox{\rm Ker } \psi = Q$.
\end{lemma}

\begin{proof}
It can be easily seen that $\psi (Q) = 0$ so $Q \subseteq
\mbox{\rm Ker } \psi $. Therefore, one can define a homomorphism
$\widehat{\psi} : H/Q \rightarrow U/(U_1+U_2')$ by $
\widehat{\psi} (h + Q) = \psi (h)$ for each $h \in H$. Let
\begin{gather*}
{\mathcal C}_2 = \Big\{ h_{i_1 i_2 i_3 5 i_4}  + Q \mid \{ i_1,
i_2, i_3, i_4 \} = \{ 1,2,3,4 \} , i_1 > i_2, i_3 \Big\} ,
\\
{\mathcal C}_3 = \Big\{ h_{5 i_3 i_4 i_1 i_2 }  + Q \mid \{ i_1,
i_2, i_3, i_4 \} = \{ 1,2,3,4 \} , i_1 > i_2 \Big\} .
\end{gather*}
and let ${\mathcal C} = {\mathcal C}_2 \cup {\mathcal C}_3$. It
can be easily checked that, modulo $Q$, each element $h_{i_1 i_2
i_3 i_4 i_5}$ can be written as a linear combination of elements
of ${\mathcal C}$ so the quotient group $H/Q$ is generated by the
set ${\mathcal C}$. Note that $\widehat{\psi} ({\mathcal C}_j) =
\overline{ {\mathcal B}}_j$ $(j=2,3)$; hence, $\widehat{\psi}
({\mathcal C})$ is a basis of the free abelian group
$U/(U_1+U_2')$. It follows that $\mbox{\rm Ker } \widehat{\psi} =
0$ so $\mbox{\rm Ker } \psi = Q$, as required.
\end{proof}

Let $P$ be the subgroup of $H$ generated by $Q$ together with all
elements
\[
h_{i_1 i_2 i_3 i_4 i_5} + h_{i_1 i_2 i_4 i_3 i_5}, \  h_{i_1 i_2
i_3 i_4 i_5} + h_{i_1 i_4 i_3 i_2 i_5}.
\]
Recall that $(\overline{R}_5 \cap \overline{T}^{(4)}_5) / (U_1 +
U_2' )$ is generated by the images of the elements
(\ref{generators2}), that is, by the images of $[x_{i_1}, x_{i_2},
x_{i_3}] [x_{i_4}, x_{i_5}] + [x_{i_1}, x_{i_2}, x_{i_4}]
[x_{i_3}, x_{i_5}] + I$ and $[x_{i_1}, x_{i_2}, x_{i_3}] [x_{i_4},
x_{i_5}] + [x_{i_1}, x_{i_4}, x_{i_3}] [x_{i_2}, x_{i_5}] + I$
where $\{ i_1, i_2, i_3, i_4, i_5 \} = \{ 1, 2, 3, 4, 5 \}$. These
images coincide with $\psi (h_{i_1 i_2 i_3 i_4 i_5} + h_{i_1 i_2
i_4 i_3 i_5})$ and $\psi (h_{i_1 i_2 i_3 i_4 i_5} + h_{i_1 i_4 i_3
i_2 i_5})$, respectively; since $Q \subset P$, we have $P =
\psi^{-1} \Big( (\overline{R}_5 \cap
\overline{T}^{(4)}_5)/(U_1+U_2')\Big) $.

To complete the proof of Proposition \ref{prop_main} we need the
following.

\begin{lemma}\label{hnotinP}
$h_{1 2 3 4 5} \notin P$.
\end{lemma}

\begin{proof}
Let $\mu : H \rightarrow \mathbb Z$ be the homomorphism of $H$
into $\mathbb Z$ defined by
\[ \mu (h_{i_1 i_2 i_3 i_4 i_5})= \mbox{sgn} (\sigma )
\]
where $\sigma = \left(
\begin{array}{ccccc}
1 & 2 & 3 & 4 & 5 \\
i_1 & i_2 & i_3 & i_4 & i_5
\end{array}
\right)$. Then
\begin{multline*}
\mu (h_{i_1 i_2 i_3 i_4 i_5} + h_{i_2 i_1 i_3 i_4 i_5}) = \mu (
h_{i_1 i_2 i_3 i_4 i_5} + h_{i_1 i_2 i_3 i_5 i_4})  \\
= \mu (h_{i_1 i_2 i_3 i_4 i_5} + h_{i_1 i_2 i_4 i_3 i_5}) = \mu (
h_{i_1 i_2 i_3 i_4 i_5} + h_{i_1 i_4 i_3 i_2 i_5}) = 0
\end{multline*}
and
\[
\mu (h_{i_1 i_2 i_3 i_4 i_5} + h_{i_2 i_3 i_1 i_4 i_5} + h_{i_3
i_1 i_2 i_4 i_5}) = \pm 3
\]
so $\mu (P) = 3 \, \mathbb Z$. On the other hand, $\mu (h_{12345})
= 1 \notin 3 \, \mathbb Z $ so $h_{12345} \notin P$.
\end{proof}

Now we are in a position to complete the proof of Proposition
\ref{prop_main}. Let
\[
\eta : U/(U_1+U_2') \rightarrow U/(\overline{R}_5 \cap
\overline{T}^{(4)}_5)
\]
be the natural epimorphism and let
\[
\overline{\psi} = \psi \circ \eta : H \rightarrow
U/(\overline{R}_5 \cap \overline{T}^{(4)}_5).
\]
Then $\mbox{\rm Ker } \overline{\psi} = \psi^{-1} \Big(
(\overline{R}_5 \cap \overline{T}^{(4)}_5)/(U_1+U_2')\Big) = P$.
It follows from Lemma \ref{hnotinP} that $\overline{\psi}
(h_{12345}) \ne 0$, that is, $[x_1,x_2,x_3][x_4,x_5] +
(\overline{R}_5 \cap \overline{T}^{(4)}_5) \ne (\overline{R}_5
\cap \overline{T}^{(4)}_5)$. Hence, $v + I =
[x_1,x_2,x_3][x_4,x_5] + I \notin (\overline{R}_5 \cap
\overline{T}^{(4)}_5)$. Since $v + I \in \overline{R}_5$, we have
$v + I \notin \overline{T}^{(4)}_5 = (T^{(4)}_5 + I)/I$, that is,
$v \notin T^{(4)}_5 + I$. The proof of Proposition \ref{prop_main}
is completed.

\section{Proof of Theorem 1.3}

Let $\mathbb Q$ be the field of rationals and let $\mathbb Q
\langle X \rangle$ be the free unitary associative $\mathbb
Q$-algebra on the free generating set $X$, $\mathbb Z \langle X
\rangle \subset \mathbb Q \langle X \rangle$. Let $T_{\mathbb
Q}^{(3,2)}$ be the ideal in $\mathbb Q \langle X \rangle$
generated by all elements $[a_1, a_2, a_3] [a_4, a_5]$ and $[a_1,
a_2, a_3, a_4]$ where $a_i \in \mathbb Q \langle X \rangle$. It is
clear that $T^{(3,2)} \subseteq T_{\mathbb Q}^{(3,2)} \cap \mathbb
Z \langle X \rangle$.

The idea of the proof of Theorem \ref{ZX/T32} is similar to one
used in \cite{BEJKL12} to prove that the additive group of
$\mathbb Z \langle X \rangle /T^{(3)}$ is free abelian. We will
define a certain set ${\mathcal D} \subset \mathbb Z \langle X
\rangle$ and prove that, on one hand, $\{ d + T_{\mathbb
Q}^{(3,2)} \mid d \in {\mathcal D} \}$ is a $\mathbb Q$-basis of
$\mathbb Q \langle X \rangle / T_{\mathbb Q}^{(3,2)}$ and, on the
other hand, $\{ d + T^{(3,2)} \mid d \in {\mathcal D} \}$
generates the additive group of $\mathbb Z \langle X \rangle /
T^{(3,2)}$. Then $\{ d + T^{(3,2)} \mid d \in {\mathcal D} \}$ is
a linearly independent generating set of the additive group of
$\mathbb Z \langle X \rangle / T^{(3,2)}$ so this additive group
is free abelian (with a basis $\{ d + T^{(3,2)} \mid d \in {\mathcal D} \}$).

The following lemma is well-known (see \cite[Theorem 4.3]{EKM09},
\cite[Lemma 1]{Gordienko07}, \cite[Lemma 3]{Latyshev65},
\cite[Theorem 1]{Volichenko78}).

\begin{lemma}\label{3x2}
For all $a_1, \dots , a_{6} \in \mathbb Q \langle X \rangle$,
\[
[a_1, a_2][a_3, a_4] [a_5, a_6] + [a_1, a_3][a_2, a_4] [a_5, a_6]
\in T^{(3,2)}_{\mathbb Q} .
\]
\end{lemma}

\begin{proof} We have
\begin{multline*}
[a_1 a_4, a_2, a_3][a_5, a_6] = [a_1 [a_4, a_2] + [a_1, a_2] a_4,
a_3] [a_5, a_6] \\ = (a_1 [a_4, a_2, a_3] + [a_1,a_3][a_4, a_2] +
[a_1, a_2][a_4, a_3] \\ + [a_1, a_2, a_3] a_4) [a_5, a_6] \in
T^{(3,2)}_{\mathbb Q}.
\end{multline*}
Since $[a_1 a_4, a_2, a_3][a_5, a_6] \in T^{(3,2)}_{\mathbb Q}$,
$a_1 [a_4, a_2, a_3] [a_5, a_6] \in T^{(3,2)}_{\mathbb Q}$ and
\[
[a_1, a_2, a_3] a_4 [a_5, a_6] =  a_4 [a_1, a_2, a_3] [a_5, a_6] +
[a_1, a_2, a_3, a_4] [a_5, a_6] \in T^{(3,2)}_{\mathbb Q},
\]
we have
\[
[a_1, a_2][a_3, a_4] [a_5, a_6] + [a_1, a_3][a_2, a_4] [a_5, a_6]
\in T^{(3,2)}_{\mathbb Q},
\]
as required.
\end{proof}

Since $[b_1,b_2][b_3,b_4] \equiv [b_3,b_4][b_1,b_2]
\pmod{T^{(3,2)}_{\mathbb Q}}$ for all $b_1,b_2,b_3,b_4 \in \mathbb
Q \langle X \rangle$, Lemma \ref{3x2} implies the following.

\begin{corollary}
For all $k \ge 3$, all $\sigma \in S_{2k}$ and all $a_1, \dots ,
a_{2k} \in \mathbb Q \langle X \rangle$, we have
\begin{multline}\label{sigma2}
[a_1, a_2] [a_3, a_4] \dots [a_{2k-1}, a_{2k}] \\
\equiv \mbox{sgn} (\sigma )[a_{\sigma (1)}, a_{\sigma (2)}]
[a_{\sigma (3)}, a_{\sigma (4)}] \dots [a_{\sigma (2k-1)},
a_{\sigma (2k)}] \pmod{T^{(3,2)}_{\mathbb Q}}.
\end{multline}
\end{corollary}

\noindent Note that the argument above shows also that the
congruence (\ref{sigma2}) holds in  $\mathbb Z \langle X \rangle$
modulo $T^{(3,2)}$.

Let $C$ be the unitary subalgebra in $\mathbb Q \langle X \rangle$
generated by all commutators $[x_{i_1}, x_{i_2}, \dots , x_{i_k}]$
where $k \ge 2$, $i_s \in \mathbb N$ for all $s$, $x_i \in X$ for
all $i$. Let
\begin{gather*}
{\mathcal D}_0' = \{ 1 \}, \quad {\mathcal D}_1' = \{
[x_{i_1},x_{i_2}] \mid i_1 < i_2 \} , \quad {\mathcal D}_2' = \{
[x_{i_1}, x_{i_2}, x_{i_3}] \mid i_1 < i_2, i_1 \le i_3 \} ,
\\
{\mathcal D}'_3 = \{ [x_{i_1}, x_{i_2}] [x_{i_3}, x_{i_4}] \mid
i_1 < i_2, i_3 < i_4, i_1 \le i_3; \mbox{ if } i_1 = i_3 \mbox{
then } i_2 \le i_4 \} ,
\\
{\mathcal D}'_4 = \{ [x_{i_1}, x_{i_2}] [x_{i_3}, x_{i_4}] \dots
[x_{i_{2 k -1}}, x_{i_{2k}}] \mid k \ge 3, i_1 < i_2 < \dots <
i_{2k} \} .
\end{gather*}
Let ${\mathcal D}' = {\mathcal D}_0' \cup {\mathcal D}_1' \cup
{\mathcal D}_2' \cup {\mathcal D}_3' \cup {\mathcal D}_4'$.

In the proof of the next lemma we make use of a result proved in
\cite{Gordienko07} and \cite{Volichenko78}.

\begin{lemma}\label{basisC}
The set $\{ d' + T^{(3,2)}_{\mathbb Q} \mid d' \in {\mathcal D}'
\}$ is a basis of the vector space $(C + T^{(3,2)}_{\mathbb Q}) /
T^{(3,2)}_{\mathbb Q}$ over $\mathbb Q$.
\end{lemma}

\begin{proof}
Note first that $C$ is spanned by the set ${\mathcal D}'$ modulo
$T^{(3,2)}_{\mathbb Q}$.

Indeed, it is clear that $C$ is spanned by $1$ and the products
$c_1 c_2 \dots c_m$ $(m \ge 1)$ where each $c_l$ is a commutator
of length $k(l) \ge 2$, $c_l = [x_{i_{l1}}, x_{i_{l2}}, \dots ,
x_{i_{l k(l)}}]$. Note that $c_i c_j = c_j c_i
\pmod{T^{(3,2)}_{\mathbb Q}}$ for all $i,j$. Further, if, for some
$l$, $c_l$ is a commutator of length $k(l) \ge 4$ then $c_1 c_2
\dots c_m \in T^{(3,2)}_{\mathbb Q}$. If, for some $l$, $c_l =
[x_{i_{l1}}, x_{i_{l2}}, x_{i_{l3}}]$ is a commutator of length
$3$ and $m >1$ then again $c_1 c_2 \dots c_m \in
T^{(3,2)}_{\mathbb Q}$. It follows that $(C + T^{(3,2)}_{\mathbb
Q}) / T^{(3,2)}_{\mathbb Q}$ is spanned by $1$, the commutators
$[x_{i_1},x_{i_2},x_{i_3}] + T^{(3,2)}_{\mathbb Q}$ $(i_s \in
\mathbb N )$ and the products $c_1 c_2 \dots c_m +
T^{(3,2)}_{\mathbb Q}$ where $m \ge 1$ and each $c_l$ is a
commutator of length $2$.

It is clear that in a generator $[x_{i_1}, x_{i_2}] +
T^{(3,2)}_{\mathbb Q}$ of $(C + T^{(3,2)}_{\mathbb Q}) /
T^{(3,2)}_{\mathbb Q}$ we can assume $i_1 < i_2$. Similarly, in a
generator $[x_{i_1}, x_{i_2}, x_{i_3}] + T^{(3,2)}_{\mathbb Q}$ we
may assume $i_1 < i_2$, $i_1 \le i_3$ and in a generator
$[x_{i_1},x_{i_2}] [x_{i_3},x_{i_4}] + T^{(3,2)}_{\mathbb Q}$ we
may assume that $i_1 < i_2$, $i_3 < i_4$ and $i_1 \le i_3$; if
$i_1 = i_3$ we may also assume $i_2 \le i_4$. Finally, it follows
immediately from (\ref{sigma2}) that in a generator
\[
[x_{i_1},x_{i_2}] [x_{i_3},x_{i_4}] \dots [x_{i_{2k-1}},
x_{i_{2k}}] + T^{(3,2)}_{\mathbb Q} \qquad (k \ge 3)
\]
we can assume $i_1 < i_2 < i_3 < \dots < i_{2k}$. Thus, the set
$\{ d' + T^{(3,2)}_{\mathbb Q} \mid d' \in {\mathcal D}' \}$ spans
the vector space $(C + T^{(3,2)}_{\mathbb Q}) / T^{(3,2)}_{\mathbb
Q}$ over $\mathbb Q$, as claimed.

Now to prove Lemma \ref{basisC} it remains to check that the set
$\{ d' + T^{(3,2)}_{\mathbb Q} \mid d' \in {\mathcal D}' \}$ is
linearly independent in $\mathbb Q \langle X \rangle /
T^{(3,2)}_{\mathbb Q}$.

Let $R(m_1, m_2, \dots , m_k)$ ($k \ge 0$, $m_k>0$ if $k>0$) be
the linear span in $\mathbb Q \langle X \rangle$ of all monomials
$x_{i_1} \dots x_{i_s}$ that contain $m_1$ variables $x_1$, $m_2$
variables $x_2$, $\dots$, $m_k$ variables $x_k$ and do not contain
any other variable. A polynomial $f \in \mathbb Q \langle X
\rangle$ is called \textit{multilinear} if $f \in R(m_1, \dots ,
m_k)$ where $k>0$ and $m_i \in \{ 0, 1 \}$ for all $i$. It follows
from \cite[Lemma 3]{Gordienko07} or from \cite[Theorem
1]{Volichenko78} that the multilinear elements of ${\mathcal D}'$
are linearly independent modulo $T^{(3,2)}_{\mathbb Q}$. Hence, it
remains to prove that each non-multilinear element of ${\mathcal
D}'$ is not equal, modulo $T^{(3,2)}_{\mathbb Q}$, to a linear
combination of other elements of ${\mathcal D}'$.

Note that $\mathbb Q \langle X \rangle$ is a direct sum of the
vector subspaces $R(m_1, m_2, \dots , m_k)$ ($k \ge 0$, $m_k>0$ if
$k>0$) and $\mathbb Q \langle X \rangle / T^{(3,2)}_{\mathbb Q}$
is a direct sum of the subspaces $(R(m_1, m_2, \dots , m_k) +
T^{(3,2)}_{\mathbb Q}) / T^{(3,2)}_{\mathbb Q}$. The
non-multilinear elements of ${\mathcal D}'$ are as follows:
\begin{enumerate}
\item $[x_{i_1}, x_{i_2}, x_{i_3}] \in {\mathcal D}_2'$ where $i_1
< i_2$ and either $i_3 = i_1$ or $i_3 = i_2$;
\item $[x_{i_1}, x_{i_2}] [x_{i_3}, x_{i_4}] \in {\mathcal D}_3'$
where $i_1 < i_2$, $i_3 < i_4$ and either $i_1 = i_3$, $i_2 \le
i_4$ or $i_1 < i_3$, $i_2 = i_3$ or $i_1 < i_3$, $i_2 = i_4$.
\end{enumerate}
Each non-multilinear element $d'$ above is the only element of
${\mathcal D}'$ that belongs to the corresponding term $R(m_1,
\dots , m_k)$. Hence, to prove that this non-multilinear element
$d'$ is not a linear combination of other elements of ${\mathcal
D}'$ modulo $T^{(3,2)}_{\mathbb Q}$ it suffices to check that $d'$
is not equal to $0$ modulo $T^{(3,2)}_{\mathbb Q}$, that is, $d'
\notin T^{(3,2)}_{\mathbb Q}$.

If $d' \in {\mathcal D}_2'$ then $d' \notin T^{(3,2)}_{\mathbb Q}$
because the elements of ${\mathcal D}_2'$ are of degree $3$ and
$T^{(3,2)}_{\mathbb Q}$ does not contain non-zero polynomials of
degree less than $4$. Therefore, it remains to check that $d'
\notin T^{(3,2)}_{\mathbb Q}$ for the non-multilinear elements of
${\mathcal D}_3'$. Since the ideal $T^{(3,2)}_{\mathbb Q}$ is
invariant under permutations of the set $X = \{ x_1, x_2,
\dots \}$ of variables, it suffices to check that
$[x_1,x_3][x_2,x_3] \notin T^{(3,2)}_{\mathbb Q}$ and $[x_1,x_2]^2
\notin T^{(3,2)}_{\mathbb Q}$.

Suppose, in order to get a contradiction, that $[x_1,x_3][x_2,x_3]
\in T^{(3,2)}_{\mathbb Q}$. Then $[x_1, x_3 + x_4][x_2,x_3 + x_4]
\in T^{(3,2)}_{\mathbb Q}$ so
\begin{multline*}
[x_1, x_3 + x_4][x_2,x_3 + x_4] - [x_1,x_3][x_2,x_3] -
[x_1,x_4][x_2,x_4] \\ = [x_1,x_3][x_2,x_4] + [x_1,x_4][x_2,x_3]
\in T^{(3,2)}_{\mathbb Q} .
\end{multline*}
However, $[x_1,x_3][x_2,x_4] + [x_1,x_4][x_2,x_3]$ is a sum of $2$
multilinear elements of ${\mathcal D}'$ and, as it was mentioned
above, the multilinear elements of ${\mathcal D}'$ are linearly
independent modulo $T^{(3,2)}_{\mathbb Q}$. This contradiction
shows that $[x_1,x_3][x_2,x_3] \notin T^{(3,2)}_{\mathbb Q}$. One
can prove in a similar way that $[x_1,x_2]^2 \notin
T^{(3,2)}_{\mathbb Q}$ as well.

This completes the proof of Lemma \ref{basisC}.
\end{proof}

An ideal $T$ in $\mathbb Q \langle X \rangle$ is called a
\textit{$T$-ideal} if $\phi (T) \subseteq T$ for all endomorphisms
$\phi$ of $\mathbb Q \langle X \rangle$. It is clear that
$T^{(3,2)}_{\mathbb Q}$ is a $T$-ideal in $\mathbb Q \langle X
\rangle$. The following assertion is a particular case of a result
due to Drensky \cite{Drensky84}, see also \cite[Theorem
4.3.11]{Drenskybook}.

\begin{theorem}[\cite{Drensky84}]\label{Drensky}
Let $T$ be a $T$-ideal in $\mathbb Q \langle X \rangle$ and let
${\mathcal E} \subset \mathbb Q \langle X \rangle$ be a set such
that $\{ e + T \mid e \in {\mathcal E} \}$ is a basis of the
vector space $ (C + T)/T$ over $\mathbb Q$.  Then the set
\[
\{ x_{i_1} x_{i_2} \dots x_{i_k} e \mid k \ge 0, i_1 \le i_2 \le
\dots \le i_k, e \in {\mathcal E} \}
\]
is a $\mathbb Q$-basis of $\mathbb Q \langle X \rangle /T$.
\end{theorem}

Let
\[
{\mathcal D} = \{ x_{i_1} x_{i_2} \dots x_{i_k} d' \mid k \ge 0,
i_1 \le i_2 \le \dots \le i_k, d' \in {\mathcal D}' \} .
\]
By Theorem \ref{Drensky}, Lemma \ref{basisC} implies the
following.

\begin{corollary}\label{basis32Q}
The set $\{ d + T^{(3,2)}_{\mathbb Q} \mid d \in {\mathcal D} \}$
is a basis of $\mathbb Q \langle X \rangle / T^{(3,2)}_{\mathbb
Q}$ over $\mathbb Q$.
\end{corollary}

Theorem \ref{ZX/T32} is an immediate corollary of the following
assertion.

\begin{lemma}
The additive group of the quotient ring $\mathbb Z \langle X
\rangle / T^{(3,2)}$ is a free abelian group with a basis $\{ d +
T^{(3,2)} \mid d \in {\mathcal D} \}$.
\end{lemma}
\begin{proof}
It is straightforward to check that, modulo $T^{(3,2)}$, each
element of $\mathbb Z \langle X \rangle$ can be written as a
linear combination of elements of ${\mathcal D}$. Hence, the
set$\{ d + T^{(3,2)} \mid d \in {\mathcal D} \}$ generates the
additive group of $\mathbb Z \langle X \rangle / T^{(3,2)}$.

On the other hand, by Corollary \ref{basis32Q}, the elements of
${\mathcal D}$ are linearly independent modulo $T^{(3,2)}_{\mathbb
Q} \cap \mathbb Z \langle X \rangle$. Since $T^{(3,2)} \subseteq
T^{(3,2)}_{\mathbb Q} \cap \mathbb Z \langle X \rangle$, the
elements of ${\mathcal D}$ are linearly independent modulo
$T^{(3,2)}$ as well (and, since they generate $\mathbb Z \langle X
\rangle$ modulo $T^{(3,2)}$, we have $T^{(3,2)} =
T^{(3,2)}_{\mathbb Q} \cap \mathbb Z \langle X \rangle$).

Thus, the set $\{ d + T^{(3,2)} \mid d \in {\mathcal D} \}$ is a
linearly independent generating set of the additive group of
$\mathbb Z \langle X \rangle / T^{(3,2)}$ so the latter group is
free abelian with a basis $\{ d + T^{(3,2)} \mid d \in {\mathcal
D} \}$.
\end{proof}
The proof of Theorem \ref{ZX/T32} is completed.

\section{Proof of Proposition 1.4}

Let
\[
S = \{ [x_{j_1}, x_{j_2}, x_{j_3}][ x_{j_4}, x_{j_5}] \mid j_1 <
j_2 < j_3 < j_4 < j_5 \}
\]
and let $I^{(3,2)}$ be the ideal in $\mathbb Z \langle X \rangle$
generated by $S$.

\begin{lemma}\label{T32}
$T^{(3,2)} = T^{(4)} + I^{(3,2)}$.
\end{lemma}

\begin{proof}
Since $S \subset T^{(3,2)}$, we have $T^{(4)} + I^{(3,2)} \subset
T^{(3,2)}$. Hence, it remains to check that $T^{(3,2)} \subset
T^{(4)} + I^{(3,2)}$. Since, modulo $T^{(4)}$, the ideal
$T^{(3,2)}$ is generated by the elements $[a_1, a_2, a_3] [a_4,
a_5]$ $(a_i \in \mathbb Z \langle X \rangle )$, it suffices to
check that $[a_1, a_2, a_3] [a_4, a_5] \in T^{(4)} + I^{(3,2)}$
for all $a_i$. Clearly, one can assume that $a_i$ are monomials.

We use induction on degree of the polynomial $f = [a_1, a_2, a_3]
[a_4, a_5]$. If $\sum_{i=1}^5 \mbox{deg } a_i = 5$ then $f =
[x_{i_1}, x_{i_2}, x_{i_3}][ x_{i_4}, x_{i_5}]$. If $i_1, \dots ,
i_5$ are all distinct then, by (\ref{permutation2}), $f \equiv s
\pmod{T^{(4)}}$ for some $s \in S$ so $f \in T^{(4)} + I^{(3,2)}$.
If, on the other hand, $i_k = i_l$ for some $k \ne l$ then, by
(\ref{permutation2}), $f \equiv -f \pmod{T^{(4)}}$, that is, $2 f
\in T^{(4)}$. Since, by Lemma \ref{in}, $3 f \in T^{(4)}$, we have
$f \in T^{(4)}$.

Now suppose that $\sum_{i=1}^5 \mbox{deg } a_i = k > 5$. Suppose
that for all monomials $b_1, \dots , b_5 \in \mathbb Z \langle X
\rangle$ such that $\sum_{i=1}^5 \mbox{deg } b_i  < k$ it has been
already proved that $[b_1, b_2, b_3] [b_4, b_5] \in T^{(4)} +
I^{(3,2)}$. For some $i$, $1 \le i \le 5$, we have $a_i = a_i'
a_i''$ where $\mbox{deg }a_i', \mbox{deg }a_i'' < \mbox{deg }a_i$.
By (\ref{permutation2}), we may assume without loss of generality
that $i = 5$. Then
\begin{multline*}
[a_1, a_2, a_3] [a_4, a_5] = [a_1, a_2, a_3] [a_4, a_5' a_5''] =
[a_1, a_2, a_3] a_5'[a_4, a_5''] \\ + [a_1, a_2, a_3] [a_4, a_5']
a_5'' \equiv a_5'[a_1, a_2, a_3] [a_4, a_5''] + [a_1, a_2, a_3]
[a_4, a_5'] a_5'' \pmod{T^{(4)}}.
\end{multline*}
By the inductive hypothesis, $[a_1, a_2, a_3] [a_4, a_5''], [a_1,
a_2, a_3] [a_4, a_5'] \in  T^{(4)} + I^{(3,2)}$ so $[a_1, a_2,
a_3] [a_4, a_5] \in T^{(4)} + I^{(3,2)}$, as required.
\end{proof}

Let $\xi : \mathbb Z \langle X \rangle \rightarrow \mathbb Z
\langle X_4 \rangle$ be the projection such that $\xi (x_i) = x_i$
if $i \le 4$ and $\xi (x_i) = 0$ otherwise. Note that $\xi (
T^{(3,2)} ) = T_4^{(3,2)}$, $\xi (T^{(4)}) = T_4^{(4)}$. Since
$\xi (S) = 0$, we have $\xi (I^{(3,2)}) = 0$ so, by Lemma
\ref{T32},
\[
T_4^{(3,2)} = \xi (T^{(3,2)}) = \xi (T^{(4)} + I^{(3,2)}) = \xi
(T^{(4)}) + \xi (I^{(3,2)}) = \xi (T^{(4)}) = T_4^{(4)}.
\]
It follows that $T_m^{(3,2)} = T_m^{(4)}$ for all $m$ such that $2
\le m \le 4$. The proof of Proposition \ref{T44T324} is completed.

\bigskip \noindent
\textbf{Acknowledgements.} The author is grateful to Alexander
Grishin, Plamen Koshlukov and Dimas Jos\'e Gon\c{c}alves for
useful discussions. The author was partially supported by CNPq, by
DPP/UnB, by CNPq-FAPDF PRONEX grant 2009/00091-0
(193.000.580/2009) and by RFBR grant 11-01-00945.

\end{document}